\documentclass[11pt]{article}
\usepackage{amsmath}
\usepackage{amsthm}
\usepackage{latexsym}
\usepackage[noblocks]{authblk}
\usepackage{amssymb,amsmath}  % use mathematical symbols
\usepackage{multimedia}
\usepackage{subfigure}
\usepackage{cite}
\usepackage{authblk} %%% new line between authors
\usepackage{mathrsfs} \usepackage{epsfig}
\usepackage{mathdots}
\usepackage{caption}
\usepackage{color}
\usepackage{enumerate}
\usepackage{algorithm}
\usepackage{algorithmic}
\usepackage{pgfpages,tikz,tikz-cd,pgfkeys,pgfplots}
\usetikzlibrary{arrows,positioning,matrix,fit,backgrounds,shapes}
\usepackage{lineno}

\usetikzlibrary{arrows,positioning,matrix,fit,backgrounds,shapes,shapes.geometric, intersections}
\usetikzlibrary{shapes.callouts,decorations.pathmorphing}
\usetikzlibrary{shadows} % shadow
\usetikzlibrary{decorations.markings} %%put arrowheads in the middle of directed edges
\tikzstyle{vertex}=[circle, draw, inner sep=0pt, minimum size=6pt] % style

\definecolor{LemonChiffon}{rgb}{100, 98, 80}
\definecolor{myblue}{rgb}{0,0.4,0.8}
\definecolor{orange}{rgb}{1, 0.4, 0}
\definecolor{mygreen}{rgb}{0, 0.8, 0.2}
\definecolor{myred}{rgb}{204, 0, 0}
\definecolor{violet}{RGB}{0.4,0.2,1}
\definecolor{brown}{rgb}{0.6, 0.4, 0}

\newtheorem{theorem}{Theorem}[section]
\newtheorem{lemma}[theorem]{Lemma}

\theoremstyle{definition}

\theoremstyle{definition}

\makeatletter
\newcommand{\ALOOP}[1]{\ALC@it\algorithmicloop\ #1%
  \begin{ALC@loop}}
\newcommand{\ENDALOOP}{\end{ALC@loop}\ALC@it\algorithmicendloop}

\makeatother

\begin{document}

\title{Planarity of generalized ladder graphs}
\author{Hojin Chu$^{1}$, Suh-Ryung Kim$^{1}$, Homoon Ryu$^{1}$ \\
 {\footnotesize $^{1}$ \textit{Department of Mathematics Education,
Seoul National University,}}\\{\footnotesize\textit{
Seoul 08826, Rep. of Korea}}\\
{\footnotesize\textit{
ghwls8775@snu.ac.kr, srkim@snu.ac.kr, ryuhomun@naver.com}}\\
{\footnotesize}}
\date{}
\maketitle

\begin{abstract}
The Cartesian product of $P_2$ and $P_n$ is called an $n$-ladder graph for a positive integer $n$.
%A ladder graph is an $n$-ladder graph for some integer $n$.
We call two paths $P_m$ and $P_n$ together with some edges each of which joins a vertex on $P_m$ and a vertex on $P_n$ a generalized $(m,n)$-ladder graph.
In this paper, we completely characterize the planar generalized ladder graphs and the outerplanar generalized ladder graphs.
A functigraph $C(P_n,f)$ is a generalized $(n,n)$-ladder graph.
Consequently, our result solves the problem posed by A. Chen {\it et al.} (2011) to characterize planar functigraphs $C(P_n,f)$.
%Further, we provide an $O(n)$-time algorithm for planarity testing of a generalized ladder graph of order $n$ with at most $3n-6$ edges, which is a necessary condition for a planar graph.
 \end{abstract}
{\small
 }

 \section{Introduction}
 In this paper, all the graphs are assumed to be undirected and simple.
 For a graph $G$, we call $V(G)$ the {\it vertex set} of $G$ and $E(G)$ the {\it edge set} of $G$.

%linear time planarity testing algorithm을 처음 제시한 논문에 쓰인 표현
%An aㅓbitrary (undirected) graph with V vertices may have as many as E = V(V -- 1)/'2 edges. However, a planar graph has E _< 3V - 3 by Lemma 1. Thus it may be possible to devise a planarity algorithm with a time bound which is linear in the number of vertices.

We denote a path of length $n-1$ by $P_n$ for a positive integer $n$.
The Cartesian product of $P_2$ and $P_n$ is called an {\it $n$-ladder graph} for a positive integer $n$.
%A {\it ladder graph} is an $n$-ladder graph for some positive integer $n$.
We call two paths $P_m$ and $P_n$ together with some edges each of which joins a vertex on $P_m$ and a vertex on $P_n$ a {\it generalized $(m,n)$-ladder graph}.
We say that a graph is a {\it generalized ladder graph} if it is a generalized $(m,n)$-ladder graph for some positive integers $m$ and $n$.

Given a generalized $(m,n)$-ladder graph $G$, we denote the two paths in $G$ by $G_1:=u_1u_2\cdots u_m$ and $G_2:=v_1v_2\cdots v_n$.
In addition, let $[G_1,G_2]$ denote the set of edges each of which joins a vertex in $G_1$ and a vertex in $G_2$.
For each edge $e$ in $[G_1,G_2]$, we denote by $l(e)$ and $r(e)$ the indices of the end of $e$ on $G_1$ and the end of $e$ on $G_2$, respectively.
Furthermore,
we define the following edge sets for each edge $e$ in $[G_1,G_2]$:
\[L_G^{\uparrow}(e)=\{e'\in [G_1,G_2] \colon\, l(e)<l(e')\le m \}; \quad L_G^{\downarrow}(e)=\{e'\in [G_1,G_2] \colon\, 1\le l(e')< l(e) \}; \]
\[R_G^{\uparrow}(e)=\{e'\in [G_1,G_2] \colon\, r(e)< r(e')\le n \}; \quad R_G^{\downarrow}(e)=\{e'\in [G_1,G_2] \colon\, 1 \le r(e') < r(e)\}.\]

A graph $G$ is said to be {\it planar} if there is a planar embedding which is isomorphic to $G$ and its edges intersect only at their ends.
 A graph $G$ is said to be {\it outerplanar} if there is a planar embedding for which all the vertices belong to the outer face.
 The following two theorems are well-known results for the planarity and the outerplanarity.
\begin{theorem}[Kuratowski \cite{grph}]\label{thm:kur} A graph is planar if and only if it contains no subdivision of either $K_5$ or $K_{3,3}$.
\end{theorem}

\begin{theorem}[Chartrand and Harary \cite{outpl}]\label{thm:outpl}
A graph is outerplanar if and only if it contains no subdivision of either $K_4$ or $K_{3,2}$.
\end{theorem}

In this paper, we present the following theorems that determine the planarity and the outerplanarity of generalized ladder graphs.

\begin{theorem}\label{thm:planar}
Let $G$ be a generalized ladder graph.
Then $G$ is planar if and only if for each edge $e \in [G_1, G_2]$, at least one of the following edge sets is empty:
\[ L_G^{\uparrow}(e) \cap R_G^{\downarrow}(e); \quad
L_G^{\uparrow}(e) \cap  R_G^{\uparrow}(e); \quad
L_G^{\downarrow}(e) \cap R_G^{\uparrow}(e); \quad
L_G^{\downarrow}(e) \cap R_G^{\downarrow}(e).\]
\end{theorem}

\begin{theorem}\label{thm:outerplanar}
Let $G$ be a generalized ladder graph.
Then $G$ is outerplanar if and only if one of the following holds:
\begin{itemize}
\item[(i)] $L_G^{\uparrow}(e) \cap R_G^{\uparrow}(e)=L_G^{\downarrow}(e) \cap R_G^{\downarrow}(e)=\emptyset$ for each edge $e \in [G_1, G_2]$;
\item[(ii)] $L_G^{\uparrow}(e) \cap R_G^{\downarrow}(e)=L_G^{\downarrow}(e) \cap R_G^{\uparrow}(e)=\emptyset$ for each edge $e \in [G_1, G_2]$.
\end{itemize}
\end{theorem}

A.\ Chen {\it et al.} \cite{functi} introduced the notion of functigraphs.
Let $G'$ and $G''$ be two copies of a graph $G$ with disjoint vertex sets.
In addition, let $f$ be a function from $V(G')$ to $V(G'')$.
The {\it functigraph} $C(G,f)$ is defined to be the graph with the vertex set
    $V(C(G,f))=V(G') \cup V(G'')$
    and the edge set
    $E(C(G,f))=E(G')\cup E(G'') \cup \{uv \colon\, u\in V(G'),\ v\in V(G''),\ v=f(u) \}$.
In the same paper, they proposed a problem to characterize planar functigraphs $C(P_n,f)$ and Theorem~\ref{thm:planar} solves it
since a functigraph $C(P_n,f)$ is a generalized $(n,n)$-ladder graph.

While proving Theorem~\ref{thm:planar}, we give a specific planar embedding of a given generalized ladder graph if it is planar.
For instance, for a generalized $(15,13)$-ladder graph given in Figure~\ref{fig:intro1}, its planar embedding given in Figure~\ref{fig:intro2} is obtained in the proof of Theorem~\ref{thm:planar}.
%Further, based upon Theorem~\ref{thm:planar}, we present an $O(n)$-time algorithm to check whether or not a given generalized ladder graph of order $n$ with at most $3n-6$ edges is planar (Section~\ref{sec:alg}).

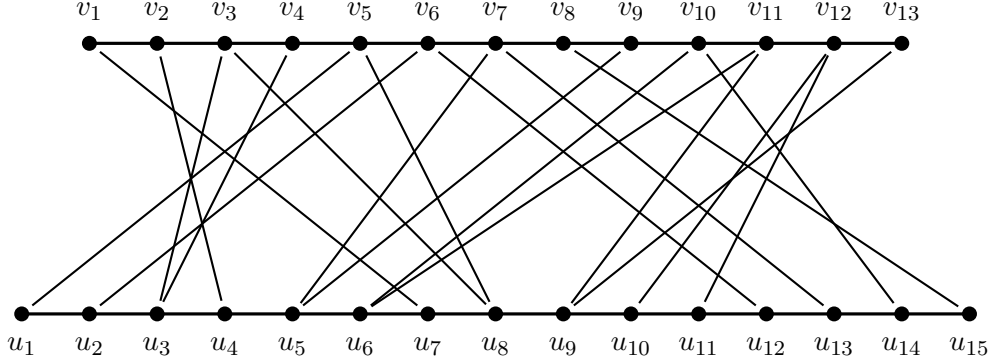
\begin{figure}
        \begin{center}
        \begin{tikzpicture}[scale=0.9]
        \tikzset{mynode/.style={inner sep=2pt,fill,outer sep=2.3pt,circle}}
        \node [mynode] (u1) at (-7,-2) [label=below :$u_{1}$] {};
        \node [mynode] (u2) at (-6,-2) [label=below :$u_{2}$] {};
        \node [mynode] (u3) at (-5,-2) [label=below :$u_{3}$] {};
        \node [mynode] (u4) at (-4,-2) [label=below :$u_{4}$] {};
        \node [mynode] (u5) at (-3,-2) [label=below :$u_{5}$] {};
        \node [mynode] (u6) at (-2,-2) [label=below :$u_{6}$] {};
        \node [mynode] (u7) at (-1,-2) [label=below :$u_{7}$] {};
        \node [mynode] (u8) at (0,-2) [label=below :$u_{8}$] {};
        \node [mynode] (u9) at (1,-2) [label=below :$u_{9}$] {};
        \node [mynode] (u10) at (2,-2) [label=below :$u_{10}$] {};
        \node [mynode] (u11) at (3,-2) [label=below :$u_{11}$] {};
        \node [mynode] (u12) at (4,-2) [label=below :$u_{12}$] {};
        \node [mynode] (u13) at (5,-2) [label=below :$u_{13}$] {};
        \node [mynode] (u14) at (6,-2) [label=below :$u_{14}$] {};
        \node [mynode] (u15) at (7,-2) [label=below :$u_{15}$] {};
        \node [mynode] (v1) at (-6,2) [label=above :$v_{1}$] {};
        \node [mynode] (v2) at (-5,2) [label=above :$v_{2}$] {};
        \node [mynode] (v3) at (-4,2) [label=above :$v_{3}$] {};
        \node [mynode] (v4) at (-3,2) [label=above :$v_{4}$] {};
        \node [mynode] (v5) at (-2,2) [label=above :$v_{5}$] {};
        \node [mynode] (v6) at (-1,2) [label=above :$v_{6}$] {};
        \node [mynode] (v8) at (0,2) [label=above :$v_{7}$] {};
        \node [mynode] (v9) at (1,2) [label=above :$v_{8}$] {};
        \node [mynode] (v10) at (2,2) [label=above :$v_{9}$] {};
        \node [mynode] (v11) at (3,2) [label=above :$v_{10}$] {};
        \node [mynode] (v12) at (4,2) [label=above :$v_{11}$] {};
        \node [mynode] (v14) at (5,2) [label=above :$v_{12}$] {};
        \node [mynode] (v15) at (6,2) [label=above :$v_{13}$] {};
        \draw[-, very thick] (-7,-2) -- (7,-2);
        \draw[-, very thick] (-6,2) -- (6,2);
        \draw[-, thick] (u1) edge (v5);
        \draw[-, thick] (u2) edge (v6);
        \draw[-, thick] (u3) edge (v3);%
        \draw[-, thick] (u3) edge (v4);
        \draw[-, thick] (u4) edge (v2);
        \draw[-, thick] (u5) edge (v10);
        \draw[-, thick] (u5) edge (v8);%
        \draw[-, thick] (u6) edge (v11);
        \draw[-, thick] (u6) edge (v12);%
        \draw[-, thick] (u7) edge (v1);
        \draw[-, thick] (u8) edge (v3);%
        \draw[-, thick] (u8) edge (v5);
        \draw[-, thick] (u9) edge (v12);
        \draw[-, thick] (u9) edge (v15);%
        \draw[-, thick] (u10) edge (v14);
        \draw[-, thick] (u11) edge (v14);
        \draw[-, thick] (u12) edge (v6);
        \draw[-, thick] (u13) edge (v8);
        \draw[-, thick] (u14) edge (v11);
        \draw[-, thick] (u15) edge (v9);
        \end{tikzpicture}
        \end{center}
        \caption{A generalized $(15,13)$-ladder graph $G$}
        \label{fig:intro1}
       \end{figure}

\begin{figure}
        \begin{center}
        \begin{tikzpicture}[scale=1]
        \tikzset{mynode/.style={inner sep=2pt,fill,outer sep=2.3pt,circle}}
        \node [mynode] (u1) at (-3,-3.5) [label=right :$u_{1}$] {};
        \node [mynode] (u2) at (-3,-3) [label=right :$u_{2}$] {};
        \node [mynode] (u3) at (-3,-2.5) [label=right :$u_{3}$] {};
        \node [mynode] (u4) at (-3,-2) [label=right :$u_{4}$] {};
        \node [mynode] (u5) at (-3,-1.5) [label=right :$u_{5}$] {};
        \node [mynode] (u6) at (-3,-1) [label=right :$u_{6}$] {};
        \node [mynode] (u7) at (-3,-0.5) [label=right :$u_{7}$] {};
        \node [mynode] (u8) at (-3,0) [label=right :$u_{8}$] {};
        \node [mynode] (u9) at (-3,0.5) [label=right :$u_{9}$] {};
        \node [mynode] (u10) at (-3,1) [label=right :$u_{10}$] {};
        \node [mynode] (u11) at (-3,1.5) [label=right :$u_{11}$] {};
        \node [mynode] (u12) at (-3,2) [label=right :$u_{12}$] {};
        \node [mynode] (u13) at (-3,2.5) [label=right :$u_{13}$] {};
        \node [mynode] (u14) at (-3,3) [label=right :$u_{14}$] {};
        \node [mynode] (u15) at (-3,3.5) [label=right :$u_{15}$] {};
        \node [mynode] (v1) at (3,3) [label=left :$v_{1}$] {};
        \node [mynode] (v2) at (3,2.4) [label=left :$v_{2}$] {};
        \node [mynode] (v3) at (3,1.8) [label=left :$v_{3}$] {};
        \node [mynode] (v4) at (3,1.2) [label=left :$v_{4}$] {};
        \node [mynode] (v5) at (3,0.6) [label=left :$v_{5}$] {};
        \node [mynode] (v6) at (3,0) [label=left :$v_{6}$] {};
        \node [mynode] (v8) at (3,-0.6) [label=left :$v_{7}$] {};
        \node [mynode] (v9) at (3,-1.2) [label=left :$v_{8}$] {};
        \node [mynode] (v10) at (3,-1.8) [label=left :$v_{9}$] {};
        \node [mynode] (v11) at (3,-2.4) [label=left :$v_{10}$] {};
        \node [mynode] (v12) at (3,-3.0) [label=left :$v_{11}$] {};
        \node [mynode] (v14) at (3,-3.5) [label=left :$v_{12}$] {};
        \node [mynode] (v15) at (3,-4) [label=left :$v_{13}$] {};
        \draw[-, very thick] (-3,-3.5) -- (-3,3.5);
        \draw[-, very thick] (3,-4) -- (3,3);
        \draw[-, thick] (u3) edge (v4);
        \draw[-, thick] (u3) edge (v3);%
        \draw[-, thick] (u4) edge (v2);
        %\draw[dotted, thick] (3,3.5) -- (7,8.5);
        %\draw[dotted, thick] (3,3.5) -- (3,6);
        %\draw[dotted, thick] (-3,3.5) -- (-7,8.5);
        %\draw[dotted, thick] (-3,-4) -- (-5,-6.5);
        %\draw[dotted, thick] (-3,-4) -- (-3,-6.5);
        \draw[solid, thick] (u7)  -- (3,3.7) -- (3.16,3.7) -- (v1);
        \draw[solid, thick] (u8)  -- (3,4.7) -- (3.96,4.7) -- (v5);
        \draw[solid, thick] (u8)  -- (3,4.3) -- (3.64,4.3) -- (v3);%
        \draw[solid, thick] (u12)  -- (3,5.3) -- (4.44,5.3) -- (v6);
        \draw[solid, thick] (u13)  -- (3,5.5) -- (4.6,5.5) -- (v8);
        \draw[solid, thick] (u15)  -- (3,5.8) -- (4.84,5.8) -- (v9);
        \draw[solid, thick] (u1)  -- (-3.16,-4.2) -- (-3,-4.2) -- (v5);
        \draw[solid, thick] (u2)  -- (-3.4,-4.5) -- (-3,-4.5) -- (v6);
        \draw[solid, thick] (u5)  -- (-3.8,-5) -- (-3,-5) -- (v8);%
        \draw[solid, thick] (u5)  -- (-3.96,-5.2) -- (-3,-5.2) -- (v10);
        \draw[solid, thick] (u6)  -- (-4.28,-5.6) -- (-3,-5.6) -- (v11);
        \draw[solid, thick] (u6)  -- (-4.36,-5.7) -- (-3,-5.7) -- (v12);%
        \draw[solid, thick] (u9)  -- (-4.68,-6.1) -- (-3,-6.1) -- (v12);
        \draw[solid, thick] (u9)  -- (-4.84,-6.3) -- (-3,-6.3) -- (v15);%
        \draw[solid, thick] (u14)  -- (-5.4,6.5) -- (5.4,6.5) -- (v11);
        \draw[solid, thick] (u11)  -- (-6.2,7.5) -- (6.2,7.5) -- (v14);
        \draw[solid, thick] (u10)  -- (-6.6,8) -- (6.6,8) -- (v14);

        \end{tikzpicture}
        \end{center}
        \caption{A planar embedding of $G$ given in Figure~\ref{fig:intro1}}
        \label{fig:intro2}
        \end{figure}
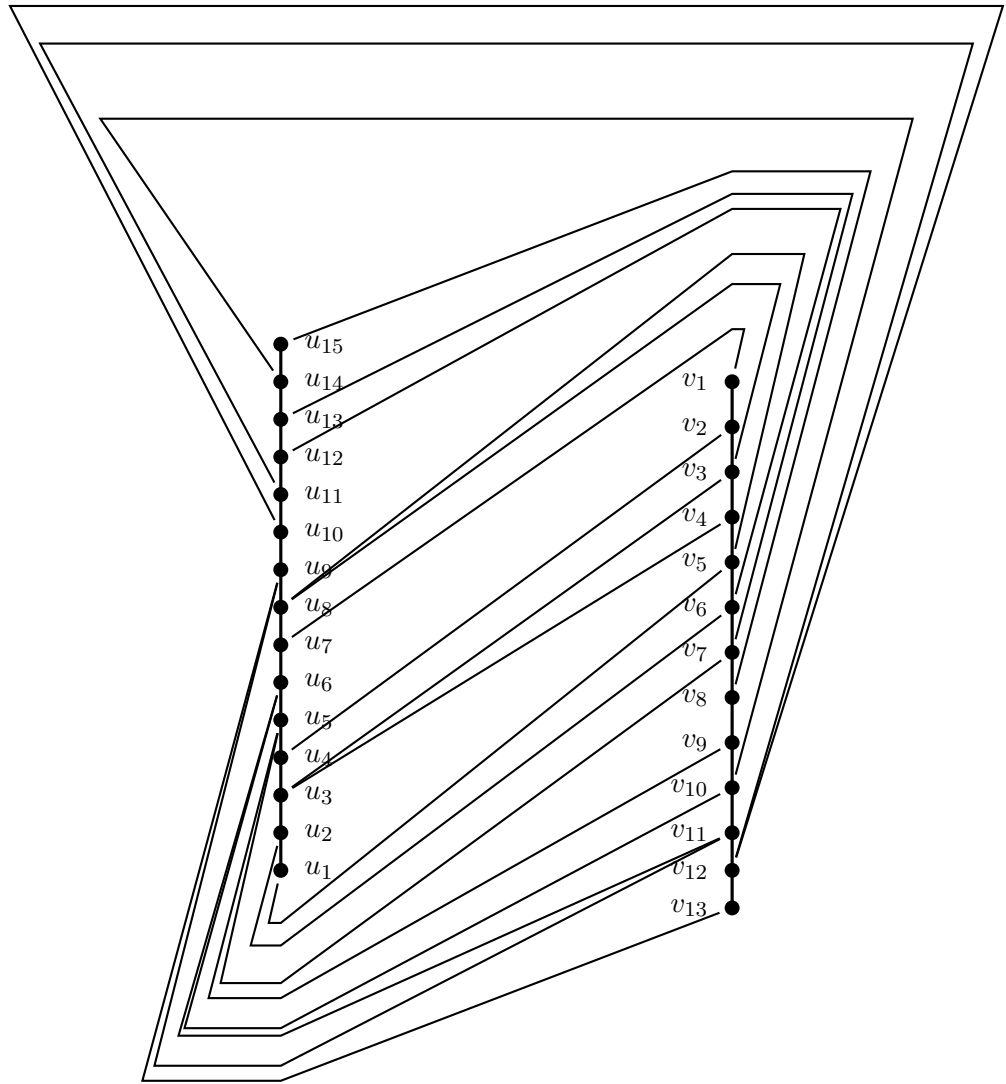

\section{Proof of Theorem~\ref{thm:planar}}

\begin{proof}[Proof of Theorem~\ref{thm:planar}]
$\Rightarrow)$
	We show the contrapositive.
	Let $G$ be a generalized $(m,n)$-ladder graph for some positive integers $m$ and $n$.
	Suppose that there is an edge $a$ in $[G_1,G_2]$ such that
	\[L_G^{\uparrow}(a) \cap R_G^{\downarrow}(a)\neq \emptyset, \quad
L_G^{\uparrow}(a) \cap R_G^{\uparrow}(a)\neq \emptyset, \quad
L_G^{\downarrow}(a) \cap R_G^{\uparrow}(a)\neq \emptyset, \quad \text{and} \quad
L_G^{\downarrow}(a) \cap R_G^{\downarrow}(a)\neq \emptyset. \]
Then there are edges $b$, $c$, $d$, and $e$ in $[G_1,G_2]$ such that
\[b \in L_G^{\uparrow}(a) \cap R_G^{\downarrow}(a), \quad
c\in L_G^{\uparrow}(a) \cap R_G^{\uparrow}(a), \quad
d\in L_G^{\downarrow}(a) \cap R_G^{\uparrow}(a), \quad \text{and} \quad
e\in L_G^{\downarrow}(a) \cap R_G^{\downarrow}(a). \]
That is, \[l(d),l(e)<l(a)<l(b),l(c) \quad \text{and} \quad  r(b),r(e)<r(a)<r(c),r(d). \]
Let \[\alpha=\min\{l(d),l(e)\}; \quad \beta=\max\{l(b),l(c)\}; \quad \gamma=\min\{r(b),r(e)\}; \quad \delta=\max\{r(c),r(d)\};\]
\[\alpha'=\max\{l(d),l(e)\}; \quad \beta'=\min\{l(b),l(c)\}; \quad \gamma'=\max\{r(b),r(e)\}; \quad \delta'=\min\{r(c),r(d)\};\]
(see Figure~\ref{fig:H} for an illustration).
Now we claim that $G$ has a $K_{3,3}$ minor.
Denote the $(u_{\alpha}, u_{\beta})$-section of $G_1$ and the $(v_{\gamma},v_{\delta})$-section of $G_2$ by $Q_1$ and $Q_2$, respectively.
We consider a subgraph $H$ of $G$ induced by the edge set \[E(Q_1) \cup E(Q_2) \cup \{a,b,c,d,e\}. \]
Let $X=\{v_{r(a)}, u_{\alpha'}, u_{\beta'}\}$ and $Y=\{u_{l(a)}, v_{\gamma'}, v_{\delta'}\}$.
We say that a path in $H$ is {\it good} if its end vertices are in $X\cup Y$ and each interior vertex of $H$ is not contained in $X\cup Y$.
Then any two good paths in $H$ are internally vertex-disjoint.
For, in $H$, the vertices in $X \cup Y$ are the only vertices of degree $3$ and the other vertices have degree $2$.
We will show that $H$ is a subdivision of $K_{3,3}$ with bipartition $(X,Y)$.
We note that in $H$, $v_{r(a)}$ and $u_{l(a)}$ are adjacent and there exist a $(u_{\alpha'},u_{l(a)})$-section of $Q_1$, a $(u_{\beta'},u_{l(a)})$-section of $Q_1$, a $(v_{r(a)},v_{\gamma'})$-section of $Q_2$, and a $(v_{r(a)},v_{\delta'})$-section of $Q_2$, each of which is a good path.

We take the edge $f$, the $(v_{\mu},v_{\nu})$-section of $Q_2$, and $(v_{\rho},v_{\sigma})$-section of $Q_2$ where
$f=u_{l(e)} v_{r(e)}$ and $(\mu,\nu,\rho,\sigma)=(r(e),\gamma',r(d),\delta')$ if $l(d) > l(e)$;
$f=u_{l(d)} v_{r(d)}$ and $(\mu,\nu,\rho,\sigma)=(r(d),\delta',r(e),\gamma')$ if $l(d) < l(e)$.
Now the $(u_{\alpha'},u_{\alpha})$-section of $Q_1$, the edge $f$, and $(v_{\mu},v_{\nu})$-section of $Q_2$ form a good $(u_{\alpha'},v_{\nu})$-path.
Moreover, the edge $u_{\alpha'}v_{\rho}$ and the $(v_{\rho},v_{\sigma})$-section of $Q_2$ form a good $(u_{\alpha'},v_{\sigma})$-path.
Thus we may obtain a good $(u_{\alpha'},v_{\gamma'})$-path and a good $(u_{\alpha'},v_{\delta'})$-path.

%{\it Case 1.} $\alpha=l(d)$.
%Then $\alpha'=l(e)$.
%Now the $(u_{l(e)},u_{l(d)})$-section of $Q_1$, the edge $u_{l(d)} v_{r(d)}$, and the $(v_{r(d)}, v_{\gamma'})$-section of $Q_2$ form a good $(u_{\alpha'},v_{\gamma'})$-path.
%Moreover, the edge $u_{l(e)}v_{r(e)}$ and the $(v_{r(e)}, v_{\delta'})$-section of $Q_2$ form a good $(u_{\alpha'},v_{\delta'})$-path.

%{\it Case 2.}
%$\alpha=l(e)$.
%Then $\alpha'=l(d)$.
%Now the $(u_{l(d)},u_{l(e)})$-section of $Q_1$, the edge $u_{l(e)} v_{r(e)}$, and the $(v_{r(e)}, v_{\delta'})$-section of $Q_2$ form a good $(u_{\alpha'},v_{\delta'})$ path.
%Moreover, the edge $u_{l(d)}v_{r(d)}$ and the $(v_{r(d)}, v_{\gamma'})$-section of $Q_2$ form a good $(u_{\alpha'},v_{\gamma'})$ path.

By symmetry, there exist a good $(u_{\beta'}, v_{\gamma'})$-path and a good $(u_{\beta'}, v_{\delta'})$-path in $H$ whether $\beta'=l(b)$ or $\beta'=l(c)$.
Consequently, we have shown that $H$ is a subdivision of $K_{3,3}$ with bipartition $(X,Y)$.
Hence, by Theorem~\ref{thm:kur}, $G$ is not planar.

  \begin{figure}
        \begin{center}
        \begin{tikzpicture}[scale=0.7]
        \tikzset{mynode/.style={inner sep=1.3pt,fill,outer sep=-1.5pt,circle}}
        \node [mynode] (alpha) at (-2,-3.5) [label=left :$u_{\alpha}$] {};
        \node [mynode] (alpha') at (-2,-1.5) [label=left :$u_{\alpha'}$] {};
        \node [mynode] (la) at (-2,0) [label=left :$u_{l(a)}$] {};
        \node [mynode] (beta') at (-2,2) [label=left :$u_{\beta'}$] {};
        \node [mynode] (beta) at (-2,4) [label=left :$u_{\beta}$] {};
        \node [mynode] (gamma) at (2,-4) [label=right :$v_{\gamma}$] {};
        \node [mynode] (gamma') at (2,-2) [label=right :$v_{\gamma'}$] {};
        \node [mynode] (ra) at (2,0) [label=right :$v_{r(a)}$] {};
        \node [mynode] (delta') at (2,2.5) [label=right :$v_{\delta'}$] {};
        \node [mynode] (delta) at (2,3.5) [label=right :$v_{\delta}$] {};
        \node[] at (-4,0) (g1) {$Q_1$};
        \node[] at (4,0) (g2) {$Q_2$};
        \draw[solid, thick] (-2,-3.5) -- (-2,4);
        \draw[solid, thick] (2,-4) -- (2,3.5);
        \draw[solid, thick] (la) -- (ra);
        \draw[solid, thick] (alpha) -- (-1.5,-3.5);
        \draw[solid, thick] (alpha') -- (-1.5,-1.7);
        \draw[solid, thick] (beta') -- (-1.5,2.3);
        \draw[solid, thick] (beta) -- (-1.5,3.7);
        \draw[solid, thick] (gamma) -- (1.5,-4);
        \draw[solid, thick] (gamma') -- (1.5,-2.3);
        \draw[solid, thick] (delta') -- (1.5,1.5);
        \draw[solid, thick] (delta) -- (1.5,3.7);
		\end{tikzpicture}
        \end{center}
        \caption{The subgraph $H$ in Theorem~\ref{thm:planar}}
        \label{fig:H}
        \end{figure}
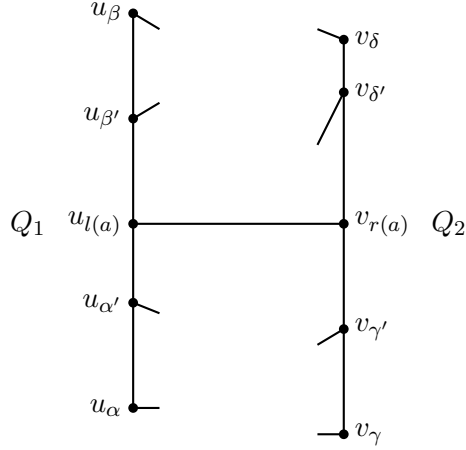

$\Leftarrow)$
Let $G$ be a generalized $(m,n)$-ladder graph for some positive integers $m$ and $n$.
Suppose that for each edge $e \in [G_1,G_2]$, at least one of the following edge sets is empty:
\[ L_G^{\uparrow}(e) \cap R_G^{\downarrow}(e); \quad
L_G^{\uparrow}(e) \cap  R_G^{\uparrow}(e); \quad
L_G^{\downarrow}(e) \cap R_G^{\uparrow}(e); \quad
L_G^{\downarrow}(e) \cap R_G^{\downarrow}(e).\]
Let
\begin{equation}
\begin{aligned}
A=\{e \in [G_1,G_2] \colon\, L_G^{\uparrow}(e) \cap R_G^{\downarrow}(e)=\emptyset\}; &\quad
	B=\{e \in [G_1,G_2] \colon\, L_G^{\uparrow}(e) \cap  R_G^{\uparrow}(e)=\emptyset\}; \\
	C=\{e \in [G_1,G_2] \colon\, L_G^{\downarrow}(e) \cap R_G^{\uparrow}(e)=\emptyset\}; &\quad
	D=\{e \in [G_1,G_2] \colon\, L_G^{\downarrow}(e) \cap R_G^{\downarrow}(e)=\emptyset\}.
\end{aligned}
\label{a1a2b1b2}
\end{equation}
By the assumption, $[G_1,G_2] = A \cup B \cup C \cup D$, not necessarily a disjoint union.
Set
\[X=A, \quad Y=B-A, \quad Z=C-(A \cup B), \quad \text{and} \quad W=D -(A \cup B \cup C). \]
Then $[G_1,G_2]$ is a disjoint union of the sets $X$, $Y$, $Z$, and $W$.

    Now, we embed the graph $G$ in the plane as follows: (see Figure~\ref{fig:planar embedding} for an illustration)
	
	First, we plot the vertices of $G$ in the following way:
 	For each $i \in [m]$ and each $j \in [n]$, locate $u_i$ at $(0,ni)$ and locate $v_j$ at $(mn,m(n-j))$.
 	Thus for each $e\in [G_1,G_2]$,
 	\begin{equation}\label{eq:u,v}
 		u_{l(e)}\mapsto (0,n\cdot l(e)) \quad \text{and} \quad v_{r(e)}\mapsto (mn,mn-m \cdot r(e)).
 	\end{equation}
 	Then we draw a straight line from $(0,n)$ to $(0, mn)$ which represents $G_1$ and draw a straight line from $(mn,0)$ to $(mn,mn-m)$ which represents $G_2$.
 	
	Each edge in $[G_1,G_2]$ will be drawn as a line which is composed of line segments whose ends points belong to \[V(G_1)\cup V(G_2) \cup X^1 \cup X^2 \cup Y^1 \cup Y^2 \cup Z^1 \cup Z^2\]
	where
	\begin{align*}
 		X^1 &= \{(x,y)\in \mathbb{Z} \times \mathbb{Z} \colon\, x=mn,\ mn < y \le 2mn\}; \\
 		X^2 &= \{(x,y)\in \mathbb{Z} \times \mathbb{Z} \colon\, y=x,\ mn < y \le 2mn\}; \\
 		Y^1 &= \{(x,y)\in \mathbb{Z} \times \mathbb{Z} \colon\, y=-x+mn,\ 2mn< y \le 3mn\}; \\
 		Y^2 &= \{(x,y)\in \mathbb{Z} \times \mathbb{Z} \colon\, y=x,\ 2mn < y \le 3mn\}; \\
 		Z^1 &= \{(x,y)\in \mathbb{Z} \times \mathbb{Z} \colon\, y=x,\ -mn \le y <0\}; \\
 		Z^2 &= \{(x,y)\in \mathbb{Z} \times \mathbb{Z} \colon\, x=0,\ -mn \le y <0\},
 	\end{align*}
	that is,
	\begin{equation}\label{eq:points}
		\begin{aligned}
		X^1&=\{\left(mn, mn-n+in+j\right) \colon\, (i,j)\in [m] \times [n] \}; \\
		X^2&=\{\left(mn-n+in+j, mn-n+in+j\right) \colon\, (i,j)\in [m] \times [n]\}; \\
		Y^1&=\{\left(-2mn+in-j, 3mn-in+j\right) \colon\, (i,j)\in [m] \times [n]\}; \\
		Y^2&=\{\left(3mn-in+j, 3mn-in+j\right) \colon\, (i,j)\in [m] \times [n]\}; \\
		Z^1&=\{\left(n-in-j, n-in-j \right) \colon\, (i,j)\in [m] \times [n]\};\\
		Z^2&=\{\left(0, n-in-j\right) \colon\, (i,j)\in [m] \times [n]\}.
	\end{aligned}
	\end{equation}
 	
	For $k=1,2$, we denote by $x^k_{i,j}$, $y^k_{i,j}$, and $z^k_{i,j}$ the point in $X^k$, $Y^k$, and $Z^k$, respectively, corresponding to $(i,j)\in [m] \times [n]$.
	We simply write $x^k_{l(e),r(e)}$, $y^k_{l(e),r(e)}$, and $z^k_{l(e),r(e)}$ as $x^k_e$, $y^k_e$, and $z^k_e$, respectively, for each $e \in [G_1,G_2]$.
	By \eqref{eq:points},
	\begin{equation}\label{eq:x^}
		x^1_e = (mn,mn-n+n\cdot l(e)+r(e)), \quad x^2_e = (mn-n+n \cdot l(e)+r(e),mn-n+n\cdot l(e)+r(e));
	\end{equation}
	\begin{equation}\label{eq:y^}
	\begin{aligned}
		y^1_e &= (-2mn+n\cdot l(e)-r(e), 3mn-n\cdot l(e)+r(e)),\\
		y^2_e &= (3mn-n \cdot l(e)+r(e), 3mn-n \cdot l(e)+r(e));
	\end{aligned}
	\end{equation}
	\begin{equation}\label{eq:w^}
		z^1_e = (n-n \cdot l(e)-r(e), n-n \cdot l(e)-r(e)), \quad  z^2_e = (0,n-n \cdot l(e)-r(e)),
	\end{equation}
	for each $e \in [G_1,G_2]$.
	
Given a line $L$ and two points $P$, $Q$ on $L$, we mark the segment $S$ of $L$ between $P$ and $Q$ by
\[P\stackrel{S}{\text{------}}Q. \]
Now we are ready to describe the line representing $e$ in $[G_1,G_2]$ precisely.
\begin{itemize}
	\item We represent each $e$ in $X$ as a line consisting of line segments:
    \[u_{l(e)}\stackrel{\Lambda_{x,1}(e)}{\text{------}}x^1_e\stackrel{\Lambda_{x,2}(e)}{\text{------}}x^2_e\stackrel{\Lambda_{x,3}(e)}{\text{------}}v_{r(e)}.\]
    \item We represent each $e$ in $Y$ as a line consisting of line segments:
    \[u_{l(e)}\stackrel{\Lambda_{y,1}(e)}{\text{------}}y^1_e\stackrel{\Lambda_{y,2}(e)}{\text{------}}y^2_e\stackrel{\Lambda_{y,3}(e)}{\text{------}}v_{r(e)}.\]
    \item We represent each $e$ in $Z$ as a line consisting of line segments: \[u_{l(e)}\stackrel{\Lambda_{z,1}(e)}{\text{------}}z^1_e\stackrel{\Lambda_{z,2}(e)}{\text{------}}z^2_e\stackrel{\Lambda_{z,3}(e)}{\text{------}}v_{r(e)}.\]
    \item We represent each $e$ in $W$ as a straight line $\Lambda_{w}(e)$ from $u_{l(e)}$ to $v_{r(e)}$.
\end{itemize}
    We will show that there is no crossing in the above drawing.
    To show that there is no crossing in the above drawing, it suffices to check whether some two edges in $[G_1,G_2]$ cross or not.

   Take two distinct edges $e$ and $e'$ in $X$.
   	By $\eqref{eq:x^}$, we may check the following:
   	\begin{equation*}\label{eq:equiv1}
   		\begin{aligned}
   	x^1_e \text{ lies above } x^1_{e'}  &\iff mn+n(l(e)-1)+r(e) > mn+n(l(e')-1)+r(e')\\
   	&\iff \text{either } l(e)> l(e') \text{ or } l(e)=l(e') \text{ and } r(e) > r(e')     \\
   	&\iff x^2_e \text{ lies above } x^2_{e'}.
   		\end{aligned}
   	\end{equation*}
   	Further, since $r(e) \ge r(e')$ if $l(e)>l(e')$ by the definition of $X$, the second equivalence can be restated
   	\begin{equation}\label{eq:equiv2}
   	e\neq e',\ l(e)\ge l(e'), \text{ and } r(e) \ge r(e') \iff
   		x^1_e \text{ lies above } x^1_{e'}
   		\iff x^2_e \text{ lies above } x^2_{e'}.
   	\end{equation}
   	We note that for each edge $a$ in $X$,
   	\begin{align*}
   		(x,y) \in \Lambda_{x,1}(a)-\{x^1_a \} &\implies \quad 0\le x< mn; \\
   		(x,y) \in \Lambda_{x,2}(a)-\{x^2_a \} &\implies x\ge mn \text{ and } y> x; \\
   		(x,y) \in \Lambda_{x,3}(a)\quad  &\implies x\ge mn \text{ and } y\le x.
   	\end{align*}
   	Then noting that $\{0\le x< mn\}$, $\{x\ge mn, y> x\}$, and $\{x\ge mn, y\le x\}$ are mutually disjoint.
   	It is easy to check that there is no crossing for $e$ and $e'$ if and only if $\Lambda_{x,i}(e)$ and $\Lambda_{x,i}(e')$ do not cross for each $i=1,2,3$.
	One may check that $x^1_e\neq x^1_{e'}$ for $e \neq e'$.
	Without loss of generality, we may assume that $x^1_e$ lies above $x^1_{e'}$.
	Then, by \eqref{eq:equiv2},	\begin{itemize}
		\item there is no crossing for $\Lambda_{x,1}(e)$ and $\Lambda_{x,1}(e')$ since $u_{l(e')}$ does not lie above $u_{l(e)}$;
	\item there is no crossing for $\Lambda_{x,2}(e)$ and $\Lambda_{x,2}(e')$ since $x^2_e$ lies above $x^2_{e'}$.
	\end{itemize}
	Further, there is no crossing for $\Lambda_{x,3}(e)$ and $\Lambda_{x,3}(e')$.
	To see why, we consider the triangle $T$ with vertices $v_{r(e)}$, $x^2_{e}$, and $(mn,mn)$.
	Then, by \eqref{eq:equiv2}, $r(e)\ge r(e')$ and so $v_{r(e)}$ does not lie above $v_{r(e')}$ on the line $x=mn$.
	Thus $v_{r(e')}$ lies on the side of $T$ joining $v_{r(e)}$ and $(mn,mn)$.
	Moreover, $x^2_{e'}$ lies on the side of $T$ joining $x^2_{e}$ and $(mn,mn)$ since $x^2_{e'}$ lies below $x^2_{e}$ on the line $y=x$.	
	Thus $\Lambda_{x,3}(e)$ is a segment inside $T$ and so  there is no crossing for $e$ and $e'$.
	Therefore we may conclude that the edges in $X$ do not pairwise cross in the above drawing.
    By a similar argument, one can show that the same is true for each of $Y$, $Z$, and $W$.

    Now, we consider the following regions:
    \begin{align*}
    	\mathcal{R}_1&=\{(x,y) \colon\, x \ge mn \}\cup \{(x,y)\colon\, 0 \le x\le mn,\ y \ge \frac{mn-n}{mn}x+n\};
    	\\
    	\mathcal{R}_2&=\{(x,y) \colon\, x \le 0 \}\cup \{(x,y)\colon\, 0 \le x\le mn,\ y < \frac{mn-n}{mn}x+n\}.
    \end{align*}
    Take two edges $e\in X$ and $e' \in Z$.
    Then we may check that $e$ and $e'$ are included in $\mathcal{R}_1$ and $\mathcal{R}_2$, respectively.
    Since $\mathcal{R}_1 \cap \mathcal{R}_2 \subseteq \{(x,y) \colon\, x=0 \text{ or } mn\}$,
    $u_{l(e)}=u_{l(e')}$ or $v_{r(e)}=v_{r(e')}$ if $e$ and $e'$ cross.
    Therefore no edge in $X$ and no edge in $Z$ cross.
	One can show that there is no crossing for an edge in $Y$ and an edge in $W$ by a similar argument applied to the regions
	\[\{(x,y) \colon\, x \le 0 \text { or } x \ge mn \}\cup \{(x,y)\colon\, 0 \le x\le mn,\ y> mn\} \text{ and } \{(x,y) \colon\, 0\le x \le mn, y\le mn \}.\]

	Now, we consider the following regions:
    \begin{align*}
    	\mathcal{R}_3=\{(x,y) \colon\, x \ge 0 \}\cap \{(x,y)\colon\, y \le 2mn\}; \quad \mathcal{R}_4=\{(x,y) \colon\, x \le 0 \}\cup \{(x,y)\colon\, y > 2mn\}.
    \end{align*}
    Take two edges $e \in X$ and $e' \in Y$.
    Then we may check that $e$ is included in $\mathcal{R}_3$ and $\Lambda_{y,1}(e')$ and $\Lambda_{y,2}(e')$ are included in $\mathcal{R}_4$.
    Since $\mathcal{R}_3 \cap \mathcal{R}_4\subseteq \{(x,y) \colon\, x=0\}$, $e$ and $\Lambda_{y,2}(e')$ do not cross.
    Moreover, $u_{l(e)}=u_{l(e')}$ if $e$ and $\Lambda_{y,1}(e')$ cross.
    Therefore we only need to guarantee that $e$ and $\Lambda_{y,3}(e')$ do not cross.
    Since each of $u_{l(e)}$, $x^1_e$, and $x^2_e$ lies on $\{(x,y) \colon\, y\ge x \}$ by \eqref{eq:u,v} and \eqref{eq:x^}, $\Lambda_{x,1}(e)$ and $\Lambda_{x,2}(e)$ are included in $\{(x,y) \colon\, y\ge x \}$.
    On the other hand, since $y^2_{e'}$ and $v_{r(e')}$ lie on $\{(x,y) \colon\, y\le x \}$ by \eqref{eq:u,v} and \eqref{eq:y^}, $\Lambda_{y,3}(e')$ except $y^2_{e'}$ is included in $\{(x,y) \colon\, y< x \}$.
    Thus $\Lambda_{x,1}(e)$ and $\Lambda_{x,2}(e)$ do not cross with $\Lambda_{y,3}(e')$.
    To show that $\Lambda_{x,3}(e)$ and $\Lambda_{y,3}(e')$ do not cross, we claim $r(e)\le r(e')$.
    By the definitions of $X$ and $Y$, $L_G^{\uparrow}(e) \cap R_G^{\downarrow}(e)=\emptyset$ and $L_G^{\uparrow}(e') \cap R_G^{\uparrow}(e')=\emptyset$.
    If $l(e)<l(e')$, then $e' \in L_G^{\uparrow}(e)$ and so $e' \notin R_G^{\downarrow}(e)$ which implies $r(e) \le r(e')$.
    If $l(e)>l(e')$, then $e \in L_G^{\uparrow}(e')$ and so $e \notin R_G^{\uparrow}(e')$ which implies $r(e) \le r(e')$.
	In the case where $l(e)=l(e')$ and $r(e)>r(e')$,
    \[L_G^{\uparrow}(e') \cap R_G^{\downarrow}(e') =L_G^{\uparrow}(e) \cap  R_G^{\downarrow}(e') \subseteq L_G^{\uparrow}(e) \cap  R_G^{\downarrow}(e) = \emptyset\]
	and so $e' \in A=X$, which contradicts that $X$ and $Y$ are disjoint.
	Thus, if $l(e)=l(e')$, then $r(e)\le r(e')$.
	Consequently, we have shown that
	 $v_{r(e)}$ does not lie below $v_{r(e')}$.
    We also note that $y^2_{e'}$ lies above $x^2_{e}$ on the line $y=x$.
    Then, by just replacing $\Lambda_{x,3}(e')$ with $\Lambda_{y,3}(e')$ in the argument applied earlier for showing that there is no crossing of $\Lambda_{x,3}(e)$ and $\Lambda_{x,3}(e')$, we have an argument for showing that $\Lambda_{x,3}(e)$ and $\Lambda_{y,3}(e')$ do not cross.
    Therefore we may conclude that there is no crossing for an edge in $X$ and an edge in $Y$.
	By a similar argument, one may check that for edges $e \in Y$ and $e' \in Z$, $l(e) \ge l(e')$ and so
    $e$ and $e'$ do not cross.
    Thus there is no crossing for an edge in $Y$ and an edge in $Z$.

	Take two edges $e \in X$ and $e' \in W$.
	Then $L_G^{\uparrow}(e) \cap R_G^{\downarrow}(e)=\emptyset$ and $L_G^{\downarrow}(e') \cap R_G^{\downarrow}(e')=\emptyset$.
	Since each of $x^1_e$, $x^2_e$, and $v_{r(e)}$ lies on $\{(x,y) \colon\, mn\ge x \}$ by \eqref{eq:u,v} and \eqref{eq:x^}, $\Lambda_{x,2}(e)$ and $\Lambda_{x,3}(e)$ are included in $\{(x,y) \colon\, mn\ge x \}$.
    On the other hand, $\Lambda_{w}(e')$ is included in $\{(x,y) \colon\, 0\le x\le mn \}$.
    Thus $\Lambda_{x,2}(e)$ and $\Lambda_{x,3}(e)$ intersect with $\Lambda_{w}(e')$ only at $v_{r(e)}=v_{r(e')}$.
	To show that  $\Lambda_{x,1}(e)$ and $\Lambda_{w}(e')$ do not cross, we claim $l(e) \ge l(e')$.
	To the contrary, suppose $l(e)<l(e')$.
	Then $e \in L_G^{\downarrow}(e')$ and $e' \in L_G^{\uparrow}(e)$.
	Since $L_G^{\downarrow}(e') \cap R_G^{\downarrow}(e')=\emptyset$
    and $L_G^{\uparrow}(e) \cap R_G^{\downarrow}(e)=\emptyset$,
    $e \notin R_G^{\downarrow}(e)$ and $e' \notin R_G^{\downarrow}(e)$.
    Then $r(e) \ge r(e')$ and $r(e') \ge r(e)$.
    Thus $r(e)=r(e')$.
    By the assumption $l(e) < l(e')$, we have
    \[L_G^{\uparrow}(e') \cap R_G^{\downarrow}(e')=L_G^{\uparrow}(e') \cap R_G^{\downarrow}(e)\subseteq L_G^{\uparrow}(e) \cap R_G^{\downarrow}(e)=\emptyset \]
    and so $e' \in A=X$, which contradicts that $X$ and $W$ are disjoint.
    Therefore $u_{l(e')}$ does not lie lies above $u_{l(e)}$ on the line $x=0$.
	Since $x^1_e$ lies above $v_{r(e')}$ on the line $x=mn$, $\Lambda_{x,1}(e)$ and $\Lambda_{w}(e')$ do not cross.
	Thus $e$ and $e'$ do not cross.
	Therefore no edge in $X$ and no edge in $W$ cross.
	By a similar argument, one may check that for edges $e\in Z$ and an edge in $e' \in W$, $r(e)\ge r(e')$ and so $e$ and $e'$ do not cross.
	Thus there is no crossing for an edge in $Z$ and an edge in $W$.
	Hence we may conclude that the above drawing is a planar embedding of $G$.
\end{proof}

        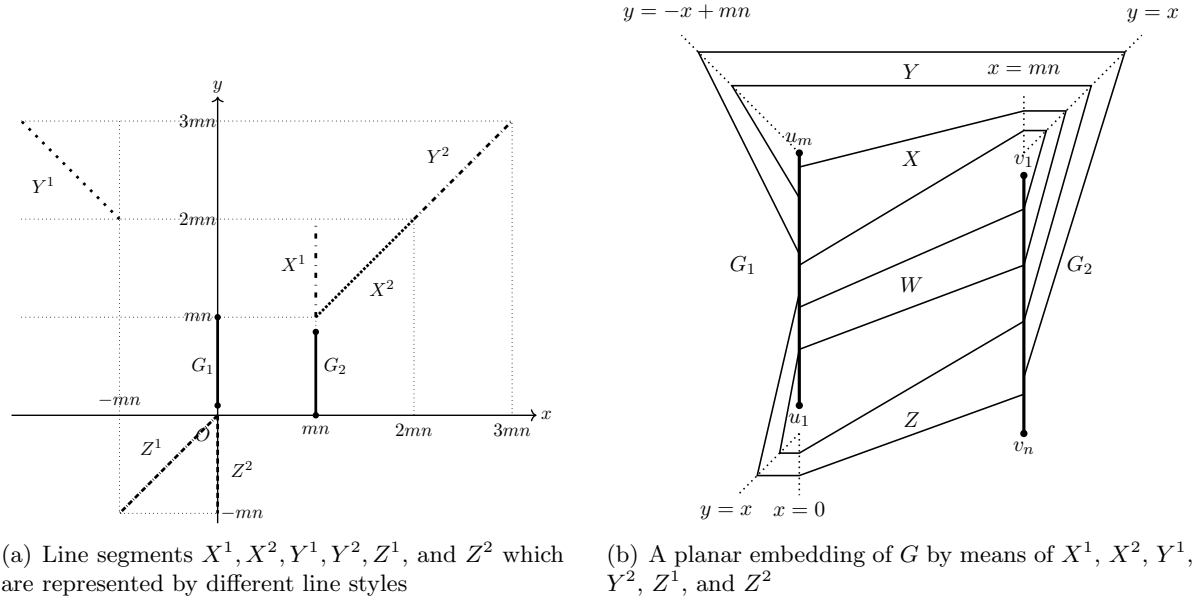
\begin{figure}
        \begin{center}
        \subfigure[Line segments $X^1, X^2, Y^1, Y^2, Z^1$, and $Z^2$ which are represented by different line styles]{
		\resizebox{0.46\textwidth}{!}{%
		\begin{tikzpicture}[scale=1]
        \tikzset{mynode/.style={inner sep=1.3pt,fill,outer sep=-1.5pt,circle}}
        \draw[->, thick] (-4.2,0) -- (6.5,0) node at (6.7,0) {$x$};
        \draw[->, thick] (0,-2.2) -- (0,6.5) node at (0,6.7) {$y$};
        \node[] at (-0.3,-0.4) {$O$};
        \node[mynode] at (0,0.2) {};
        \node[mynode] at (0,2) {};
        \node[mynode] at (2,0) {};
        \node[mynode] at (2,1.7) {};
		\draw[-, ultra thick] (0,0.2) -- (0,2);		
		\draw[-, ultra thick] (2,0) -- (2,1.7);
		\node[] at (-0.3,1) {$G_1$};		
		\node[] at (2.4,1) {$G_2$};		
		\node[] at (0.5,-2) {$-mn$};			
		\node[] at (-0.4,2) {$mn$};			
		\node[] at (-0.4,4) {$2mn$};		
		\node[] at (-0.4,6) {$3mn$};
		\node[] at (-2,0.3) {$-mn$};			
		\node[] at (2,-0.3) {$mn$};		
		\node[] at (4,-0.3) {$2mn$};
		\node[] at (6,-0.3) {$3mn$};
        \draw[dotted] (-4,2) -- (2,2) -- (2,0);
        \draw[dotted] (-4,4) -- (4,4) -- (4,0);
        \draw[dotted] (-4,6) -- (6,6) -- (6,0);
        \draw[dotted] (-2,6)-- (-2,0)  -- (-2,-2) -- (0,-2);
        %\draw[ thick, rotate=-90] (1,0) ellipse (1 and 0.2);
        %\draw[ thick, shift={(2,2)}, rotate=90] (1,0) ellipse (1 and 0.2);
		%\draw[ thick, rotate=45] (4.23,0) ellipse (1.4 and 0.3);	
		%\draw[ thick, rotate=45] (7.07,0) ellipse (1.4 and 0.3);
		%\draw[ thick, shift={(2,0)}, rotate=135] (4.25,0) ellipse (1.4 and 0.3);
		%\draw[ thick, rotate=225] (1.43,0) ellipse (1.4 and 0.3);
		\draw[loosely dashdotted, ultra thick] (2,2) -- (2,4);		
		\draw[densely dotted, ultra thick] (2,2) -- (4,4);		
		\draw[dashdotted, ultra thick] (4,4) -- (6,6);		
		\draw[loosely dotted, ultra thick] (-2,4) -- (-4,6);
		\draw[densely dashdotted, ultra thick] (0,0) -- (-2,-2);		
		\draw[densely dashed, ultra thick] (0,0) -- (0,-2);
		\node[] at (1.55,3.1) {$X^1$};		
		\node[] at (3.35,2.6) {$X^2$};
		\node[] at (-3.55,4.7) {$Y^1$};
		\node[] at (4.5,5.3) {$Y^2$};
		\node[] at (-1.35,-0.65) {$Z^1$};
		\node[] at (0.5,-1.1) {$Z^2$};
        \end{tikzpicture}
        }}
        \hspace{0.3cm}
        \subfigure[A planar embedding of $G$ by means of $X^1$, $X^2$, $Y^1$, $Y^2$, $Z^1$, and $Z^2$]{
        \resizebox{0.48\textwidth}{!}{%
        \begin{tikzpicture}[scale=1]
        \tikzset{mynode/.style={inner sep=1.3pt,fill,outer sep=-1.5pt,circle}}
        \draw[-, ultra thick] (-2,2) edge (-2,-2.5);
        \draw[-, ultra thick] (2,1.6) edge (2,-3);
        \node [mynode] (u1) at (-2,2) [label=above :$u_{m}$] {};
        \node [mynode] (un) at (-2,-2.5) [label=below :$u_1$] {};
        \node [mynode] (v1) at (2,1.6) [label=above :$v_{1}$] {};
        \node [mynode] (vn) at (2,-3) [label=below :$v_{n}$] {};
        \node[] at (-3,0) (g1) {$G_1$};
        \node[] at (3,0) (g2) {$G_2$};
        \node[] at (-0,-0.35) (a1) {$W$};
        \node[] at (-0,-2.75) (b2) {$Z$};
        \node[] at (-0,1.9)(b1) {$X$};
        \node[] at (-0,3.45) (a2) {$Y$};
        \node[] at (2,3.5) (x1) {$x=mn$};
        \node[] at (4.3,4.45) (x2) {$y=x$};
        \node[] at (-2,-4.35) (y1) {$x=0$};
        \node[] at (-4,4.5) (x2) {$y=-x+mn$};
        \node[] at (-3.3,-4.4) (x2) {$y=x$};
        \draw[dotted, thick] (-2,2) -- (-4.1,4.1);
        \draw[dotted, thick] (2,2) -- (4.1,4.1);
        \draw[dotted, thick] (2,2) -- (2,3.05);
        \draw[dotted, thick] (-2,-3) -- (-3.1,-4.1);
        \draw[dotted, thick] (-2,-3) -- (-2,-4.1);
        \draw[solid, thick] (-2,1.75)  -- (2,2.75) -- (2.75,2.75) -- (2,0);
        \draw[solid, thick] (-2,0)  -- (2,2.4) -- (2.4,2.4) -- (2,1);
        \draw[solid, thick] (-2,1.2)  -- (-3.2,3.2) -- (3.2,3.2) -- (2,-1);
        \draw[solid, thick] (-2,0.2)  -- (-3.8,3.8) -- (3.8,3.8) -- (2,-2);
        \draw[solid, thick] (-2,-0.75)  -- (2,1);
        \draw[solid, thick] (-2,-1.5)  -- (2,0);
        \draw[solid, thick] (-2,-0.5)  -- (-2.75,-3.75) -- (-2,-3.75) -- (2,-2.3);
        \draw[solid, thick] (-2,-1.5)  -- (-2.35,-3.35) -- (-2,-3.35) -- (2,-1);
        \end{tikzpicture}
        }}		
        \end{center}
        \caption{A planar embedding of $G$ in the ``if" part of Theorem~\ref{thm:planar}}
        \label{fig:planar embedding}
        \end{figure}

\section{Proof of Theorem~\ref{thm:outerplanar}}
Given a generalized $(m,n)$-ladder graph $G$ with paths $G_1=u_1 \cdots u_m$ and $G_2=v_1 \cdots v_n$,
we consider the following graphs:
\begin{itemize}
	\item the graph $G^L$ with $V(G^L)=V(G)$ and
	\[E(G^L)=E(G_1)\cup E(G_2)\cup \{u_iv_j \colon\, u_{m-i+1}v_j  \in [G_1,G_2] \};\]
	\item the graph $G^R$ with $V(G^R)=V(G)$ and
	\[E(G^R)=E(G_1)\cup E(G_2)\cup \{u_iv_j \colon\, u_{i}v_{n-j+1}  \in [G_1,G_2] \};\]
	\item the graph $G^S$ with $V(G^S)=\{u_i, v_j \colon\, 1\le i \le n, 1\le j \le m\}$ and
	\[E(G^S)=\{u_iu_{i+1} \colon\, 1\le i \le n-1 \}\cup \{v_jv_{j+1} \colon\, 1\le j \le m-1\} \cup \{u_iv_j \colon\, u_jv_i  \in [G_1,G_2] \}.\]
\end{itemize}
It is obvious that $G^L$, $G^R$, and $G^S$ are generalized ladder graphs isomorphic to $G$ such that
\begin{itemize}
	\item a generalized $(m,n)$-ladder graph $G^L$ with paths $G^L_1=u_1 \cdots u_m$ and $G^L_2=v_1 \cdots v_n$;
	\item a generalized $(m,n)$-ladder graph $G^R$ with paths $G^R_1=u_1 \cdots u_m$ and $G^R_2=v_1 \cdots v_n$;
	\item a generalized $(n,m)$-ladder graph $G^S$ with paths $G^S_1=u_1 \cdots u_n$ and $G^S_2=v_1 \cdots v_m$.
\end{itemize}
Accordingly, given a statement $\sigma$ regarding a generalized ladder graph $G$, $\sigma$ is equivalent to each of the following statements:
\begin{itemize}
	\item[(L)] the statement obtained from $\sigma$ by replacing the sets $L^\uparrow_G$ and $L^\downarrow_G$ with the sets $L^\downarrow_G$ and $L^\uparrow_G$, respectively;	
	\item[(R)] the statement obtained from $\sigma$ by replacing the sets $R^\uparrow_G$ and $R^\downarrow_G$ with the sets $R^\downarrow_G$ and $R^\uparrow_G$, respectively;
	\item[(S)] the statement obtained from $\sigma$ by replacing the sets $L^\uparrow_G$, $L^\downarrow_G$, $R^\uparrow_G$, and $R^\downarrow_G$ with the sets $R^\uparrow_G$, $R^\downarrow_G$, $L^\uparrow_G$, and $L^\downarrow_G$, respectively.
\end{itemize}

\begin{lemma}\label{lem:outerplanar}
Let $G$ be a generalized ladder graph.
Then $G$ is not outerplanar if there is an edge $e$ in $[G_1,G_2]$ satisfying the following:
\begin{enumerate}[(a)]
	\item $L_G^\uparrow(e) \cap R_G^\uparrow(e) \neq \emptyset$ or $L_G^\downarrow(e) \cap R_G^\downarrow(e) \neq \emptyset$;
	\item $L_G^\uparrow(e) \cap R_G^\downarrow(e)\neq \emptyset$ or $L_G^\downarrow(e) \cap R_G^\uparrow(e) \neq \emptyset$.
\end{enumerate}
\end{lemma}
\begin{proof}
	Let $G$ be a generalized $(m,n)$-ladder graph for some positive integers $m$ and $n$.
	Suppose that there is an edge $e$ in $[G_1,G_2]$ satisfying the hypothesis, that is, one of the following holds:
\begin{enumerate}[(i)]
	\item $L_G^\uparrow(e) \cap R_G^\downarrow(e)\neq \emptyset$ and $L_G^\uparrow(e) \cap R_G^\uparrow(e) \neq \emptyset$;	
	\item $L_G^\downarrow(e) \cap R_G^\uparrow(e) \neq \emptyset$ and $L_G^\uparrow(e) \cap R_G^\uparrow(e) \neq \emptyset$;
	\item $L_G^\uparrow(e) \cap R_G^\downarrow(e)\neq \emptyset$ and $L_G^\downarrow(e) \cap R_G^\downarrow(e) \neq \emptyset$;
	\item $L_G^\downarrow(e) \cap R_G^\uparrow(e) \neq \emptyset$ and $L_G^\downarrow(e) \cap R_G^\downarrow(e) \neq \emptyset$. \end{enumerate}
	\item
By (S), the case (i) is equivalent to the case (ii).
By (R), the case (iii) is equivalent to the case (ii).
By (L), the case (iv) is equivalent to the case (i).
Thus it suffices to show that (i) implies $G$ being not outerplanar.
Now we assume the case (iii).
Then $L_G^\uparrow(e) \cap R_G^\downarrow(e)\neq \emptyset$ and $L_G^\downarrow(e) \cap R_G^\downarrow(e) \neq \emptyset$
and so there are edges $a$ and $b$ such that
\[
a \in L^\uparrow_G (e) \cap R^\downarrow_G(e) \quad \text{and} \quad b \in L^\downarrow_G(e) \cap R^\downarrow_G(e),
\]
that is,
\[
l(b) < l(e) < l(a) \quad  \text{and} \quad \max\{r(a),r(b)\} < r(e) .
\]

We claim that $G$ has a $K_{3,2}$ minor.
Let \[\alpha=\min(r(a),r(b)) \quad \text{and} \quad \alpha'=\max(r(a),r(b)).\]
Denote the $(u_{l(b)}, u_{l(a)})$-section of $G_1$ and the $(v_{\alpha}, v_{r(e)})$-section of $G_2$ by $Q_1$ and $Q_2$, respectively.
We consider a subgraph $H$ of $G$ induced by the edge set
\[
E(Q_1) \cup E(Q_2) \cup \{a, b, e\}.
\]
Let $X = \{u_{l(a)}, u_{l(b)}, v_{r(e)}\}$ and $Y = \{u_{l(e)}, v_{\alpha'}\}$.
As we defined in the proof of Theorem~\ref{thm:planar}, we say that a path in $H$ is {\it good} if its end vertices are in $X \cup Y$ and it contains no interior vertex in $X\cup Y$.
Then any two good paths in $H$ are internally disjoint since, in $H$, the vertices in $X \cup Y$ are the only possible vertices of degree $3$ and the other vertices have degree $2$.
We will show that $H$ is a subdivision of $K_{3,2}$ with bipartition $(X,Y)$.
We note that in $H$, $v_{r(e)}$ and $u_{l(e)}$ are adjacent and there exist a $(u_{l(a)}, u_{l(e)})$-path, a $(u_{l(b)}, u_{l(e)})$-path, and a $(v_{r(e)}, v_{\alpha'})$-path each of which is a good path.

If $\alpha'=r(c)$ for some $c\in \{a,b\}$, the edge $c$ is a good $(u_{l(c)}, v_{\alpha'})$-path while the edge $d$ and the $(v_{\alpha},v_{\alpha'})$-section of $Q_2$ form a good $(u_{l(d)}, v_{\alpha'})$-path where $\{c,d\}= \{a,b\}$.
Thus there are good $(u_{l(a)}, v_{\alpha'})$-path and $(u_{l(b)}, v_{\alpha'})$-path.
Therefore we have shown that $H$ is a subdivision of $K_{3,2}$ with bipartite $(X,Y)$.
Hence, by Theorem~\ref{thm:outpl}, $G$ is not outerplanar.
\end{proof}

\begin{proof}[Proof of Theorem~\ref{thm:outerplanar}]
$\Rightarrow)$
We show the contrapositive.
Let $G$ be a generalized $(m,n)$-ladder graph for some positive integers $m$ and $n$.
Suppose that there are edges $e$ and $f$ in $[G_1,G_2]$ such that

\begin{equation}\label{eq:e,f}
		\begin{aligned}
		L_G^{\uparrow}(e) \cap R_G^{\uparrow}(e) \neq \emptyset  \quad &\text{or} \quad L_G^{\downarrow}(e) \cap R_G^{\downarrow}(e)\neq \emptyset; \\
		L_G^\uparrow(f)\cap R_G^\downarrow(f)\neq \emptyset \quad &\text{or} \quad L_G^\downarrow(f) \cap R_G^\uparrow(f) \neq \emptyset.
	\end{aligned}
	\end{equation}
If $e=f$, then $G$ is not outerplanar by Lemma~\ref{lem:outerplanar}.
Suppose that $e$ and $f$ are distinct.
If $l(e)<l(f)$ and $r(f)<r(e)$ (resp.\ $l(f)<l(e)$ and $r(e)<r(f)$), then $f \in L_G^{\uparrow}(e) \cap R_G^{\downarrow}(e)$ (resp.\ $f \in L_G^{\downarrow}(e) \cap R_G^{\uparrow}(e)$) and so $G$ is not outerplanar by Lemma~\ref{lem:outerplanar} applied to $e$.
If $l(e)<l(f)$ and $r(e)<r(f)$ (resp.\ $l(f)<l(e)$ and $r(f)<r(e)$), then $e \in L_G^{\downarrow}(f) \cap R_G^{\downarrow}(f)$ (resp.\ $e \in L_G^{\uparrow}(f) \cap R_G^{\uparrow}(f)$) and so $G$ is not outerplanar by Lemma~\ref{lem:outerplanar} applied to $f$.
Thus we conclude that if $l(e) \neq l(f)$ and $r(e) \neq r(f)$, then $G$ is not outerplanar.

Now we assume that either $l(e)=l(f)$ or $r(e)=r(f)$.
By (S), the case of $l(e)=l(f)$ is equivalent to the case of $r(e)=r(f)$.
Moreover, by (L) and the symmetry of $e$ and $f$, the case where $l(e)=l(f)$ and $r(e)<r(f)$ is equivalent to the case where $l(e)=l(f)$ and $r(f)<r(e)$.
Thus it suffices to consider the case where $l(e)=l(f)$ and $r(e) <r(f)$.

Assume that
\[l(e)=l(f) \quad \text{and} \quad r(e) <r(f). \]
Note that by \eqref{eq:e,f}, there are edges $a$ and $b$ in $[G_1,G_2]$ such that
\[a \in \left(L_G^{\uparrow}(e) \cap R_G^{\uparrow}(e)\right) \cup \left( L_G^{\downarrow}(e) \cap R_G^{\downarrow}(e)\right) \]
and
\[b \in \left(L_G^\uparrow(f)\cap R_G^\downarrow(f)\right) \cup \left( L_G^\downarrow(f) \cap R_G^\uparrow(f)\right).\]
If $a \in L_G^{\downarrow}(e) \cap R_G^{\downarrow}(e)$, then $l(a)<l(e)=l(f)$ and $r(a)<r(e)<r(f)$ and so $a \in L_G^\downarrow(f)\cap R_G^\downarrow(f)$, which implies that $G$ is not outerplanar by Lemma~\ref{lem:outerplanar} applied to $f$.
If $b \in L_G^\downarrow(f)\cap R_G^\uparrow(f)$, then $l(b)<l(f)=l(e)$ and $r(b)>r(f)>r(e)$ and so $b \in L_G^\downarrow(e)\cap R_G^\uparrow(e)$, which implies that $G$ is not outerplanar by Lemma~\ref{lem:outerplanar} applied to $e$.
Suppose that \[a \in L_G^{\uparrow}(e) \cap R_G^{\uparrow}(e) \quad \text{and} \quad b \in L_G^\uparrow(f)\cap R_G^\downarrow(f),\]
that is,
\[e \in L_G^{\downarrow}(a) \cap R_G^{\downarrow}(a) \quad \text{and} \quad f \in L_G^\downarrow(b)\cap R_G^\uparrow(b). \]
Thus, by Lemma~\ref{lem:outerplanar}, to claim that $G$ is not outerplanar, it is sufficient to show that one of the following set contains an element:
\begin{align*}
	&L_G^{\uparrow}(a) \cap R_G^{\downarrow}(a); \quad L_G^{\downarrow}(a) \cap R_G^{\uparrow}(a); \quad L_G^{\uparrow}(b) \cap R_G^{\uparrow}(b); \quad L_G^{\downarrow}(b) \cap R_G^{\downarrow}(b); \\
	&L_G^{\uparrow}(e) \cap R_G^{\downarrow}(e); \quad L_G^{\downarrow}(e) \cap R_G^{\uparrow}(e); \quad L_G^{\uparrow}(f) \cap R_G^{\uparrow}(f); \quad L_G^{\downarrow}(f) \cap R_G^{\downarrow}(f).
\end{align*}
We also note that
\[l(e)=l(f)<\min\{l(a), l(b)\}, \quad r(e)<r(a), \quad  \text{and} \quad r(b)<r(f).\]
If $r(a)<r(f)$ (resp.\ $r(b)>r(e)$), then $f\in L_G^{\downarrow}(a) \cap R_G^{\uparrow}(a)$ (resp.\ $e\in L_G^{\downarrow}(b) \cap R_G^{\downarrow}(b)$).
If $r(a)>r(f)$ (resp.\ $r(b)<r(e)$), then $a\in L_G^{\uparrow}(f) \cap R_G^{\uparrow}(f)$ (resp.\ $b\in L_G^{\uparrow}(e) \cap R_G^{\downarrow}(e)$).
Assume that $r(a)=r(f)$ and $r(b)=r(e)$.
Then
\[ r(b)=r(e)<r(a)=r(f). \]
If $l(a)<l(b)$ (resp.\ $l(b)<l(a)$), then $b \in L_G^{\uparrow}(a) \cap R_G^{\downarrow}(a)$ (resp.\ $a \in L_G^{\uparrow}(b) \cap R_G^{\uparrow}(b)$).

Suppose $l(a)=l(b)$.
Then
\[l(e)=l(f)<l(a)=l(b). \]
By the edge $a$ (resp.\ $e$) in $[G_1,G_2]$, $u_{l(a)}$ and $v_{r(a)}$ (resp.\ $u_{l(e)}$ and $v_{r(e)}$) are adjacent.
Since $l(f)=l(e)$ and $r(f)=r(a)$ (resp.\ $l(b)=l(a)$ and $r(b)=r(e)$), $u_{l(e)}$ and $v_{r(a)}$ (resp.\ $u_{l(a)}$ and $v_{r(e)}$) are adjacent by the edge $f$ (resp.\ $b$) in $[G_1,G_2]$.
Now, the edges $u_{l(a)}v_{r(a)}$, $u_{l(e)}v_{r(e)}$, $u_{l(e)}v_{r(a)}$, $u_{l(a)}v_{r(e)}$, the $(u_{l(e)},u_{l(a)})$-section of $G_1$, and the $(v_{r(e)},v_{r(a)})$-section of $G_2$ form a $K_4$-minor of $G$.
Therefore, by Theorem~\ref{thm:outpl}, $G$ is not outerplanar and this completes the proof of the ``only if" part.

$\Leftarrow)$ Let $G$ be a generalized $(m,n)$-ladder graph for some positive integers $m$ and $n$.
Suppose that one of the following holds:
\begin{itemize}
\item[(i)] $L_G^{\uparrow}(e) \cap R_G^{\downarrow}(e)=L_G^{\downarrow}(e) \cap R_G^{\uparrow}(e)=\emptyset$ for each edge $e \in [G_1, G_2]$;
\item[(ii)] $L_G^{\uparrow}(e) \cap R_G^{\uparrow}(e)=L_G^{\downarrow}(e) \cap R_G^{\downarrow}(e)=\emptyset$ for each edge $e \in [G_1, G_2]$.
\end{itemize}
By (R), the case (i) is equivalent to the case (ii).
Thus it suffices to show that (i) implies $G$ being outerplanar.
We assume the case (i).
Then
\begin{equation}\label{eq:if}
	L_G^{\uparrow}(e) \cap R_G^{\downarrow}(e)=L_G^{\downarrow}(e) \cap R_G^{\uparrow}(e)=\emptyset \quad \text{for each edge } e \in [G_1, G_2].
\end{equation}
Now we embed the graph $G$ in the plane as follows:
\begin{itemize}
	\item Plot the vertices of $G$ by locating $u_i$ and $v_j$ at $(0,i)$ and $(1,j)$, respectively, for each $i \in [m]$ and each $j\in [n]$.
	\item Draw a straight line from $(0,1)$ to $(0,m)$ which represents $G_1$ and draw a straight line from $(1,1)$ to $(1,n)$ which represents $G_2$.
	\item Draw a straight line from $u_{l(e)}$ to $v_{r(e)}$ representing each edge $e \in [G_1,G_2]$.
\end{itemize}

To show that there is no crossing in the above drawing, it suffices to check whether some two edges in $[G_1,G_2]$ cross or not.
Take two distinct edges $e$ and $e'$ in $[G_1,G_2]$.
Then $l(e) \neq l(e')$ or $r(e) \neq r(e')$.
We first consider the former case.
Without loss of generality, assume $l(e) < l(e')$.
Then $e' \in L_G^{\uparrow}(e)$.
Since $L_G^{\uparrow}(e) \cap R_G^{\downarrow}(e)=\emptyset$ by \eqref{eq:if}, $e' \not \in R_G^{\downarrow}(e)$.
Thus $r(e) \le r(e')$.
Then either $v_{r(e')}=v_{r(e)}$ or $v_{r(e')}$ lies above $v_{r(e)}$ on the line $x=1$.
Since $u_{l(e')}$ lies above $u_{l(e)}$ on the line $x=0$, just in case where $e$ and $e'$ cross, they only intersect at the point $v_{r(e')}=v_{r(e)}$.
Now we consider the latter case, that is, $r(e) \neq r(e')$.
By a similar argument, one may check that if $e$ and $e'$ intersect, they only intersect at the point $u_{l(e')}=u_{l(e)}$.
Thus there is no accidental crossing for distinct two edges in $[G_1,G_2]$.
Therefore the above drawing is a planar embedding of $G$.

It is easy to check that each edge of $G$ is contained in the region
\[\mathcal{R}=\{(x,y)\colon\, 0\le x \le 1 \}. \]
Thus the outer face contains the region $\mathbb{R}^2-\mathcal{R}$.
Since each vertex of $G$ lies on $\partial \mathcal{R}$, the outer face contains every vertex of $G$.
Hence the above drawing is an outerplanar embedding of $G$.
\end{proof}

\section{Acknowledgement}
This work was supported by Science Research Center Program through the National Research Foundation of Korea(NRF) Grant funded by the Korean Government (MSIP)(NRF-2022R1A2C\\1009648 and 2016R1A5A1008055).
Especially, the first author was supported by Basic Science Research Program through the National Research Foundation of Korea(NRF) funded by the Ministry of Education (NRF-2022R1A6A3A13063000).

\end{document}